\documentclass[11pt]{preprint}
\usepackage[font=small]{caption}
\usepackage{alltheorems,modeltheory}
\pdfoutput=1

\newcommand{\minors}{\operatorname{\nabla}}
\newcommand{\tminors}{\operatorname{\widetilde\nabla}}
\title{Nowhere dense graph classes,\\stability, and the independence property}
\author{Hans Adler and Isolde Adler}
\date{17th November 2010}
\acm{G.2.2 Graph theory; F.4.1 Mathematical logic, Model theory}

\begin{document}
\maketitle
\abstract{A class of graphs is nowhere dense if for every integer $r$ there is
a finite upper bound on the size of cliques that occur as (topological)
$r$-minors.  We observe that this tameness notion from algorithmic graph theory
is essentially the earlier stability theoretic notion of superflatness.  For
subgraph-closed classes of graphs we prove equivalence to stability and to
not having the independence property.
}

\section{Introduction}

Recently, Nešetřil and Ossona de Mendez~\cite{nwd,nwdjsl} introduced \emph{nowhere dense} classes of finite graphs,
a generalisation of many natural and important classes such as
graphs of bounded degree, planar graphs, graphs excluding a fixed minor and graphs of bounded expansion.
These graph classes play an important role in algorithmic graph theory, as
many computational problems that are hard in general become tractable when restricted to such classes.
All these graph classes are nowhere dense.
Dawar and Kreutzer~\cite{DawarK09} gave efficient algorithms for domination problems on classes of 
nowhere dense graphs. Moreover, nowhere dense classes 
were studied in the area of finite model theory~\cite{Dawar07,Dawar10} under the guise  of
\emph{(uniformly) quasi-wide} classes and again 
turn out to be well-behaved.\footnote{The equivalence of nowhere dense, quasi-wide and uniformly quasi-wide
for subgraph-closed classes was proved by Nešetřil and Ossona de Mendez~\cite{nwdjsl}.}
We observe that nowhere density is essentially the 
stability theoretic notion of superflatness introduced by Podewski and Ziegler~\cite{sg} in 1978
because of its connection to stability.

For some time we have been looking for a way to translate between tameness in finite model theory and in
stability theory. A key obstacle was the fact that tameness notions in finite model theory are generally not
even invariant under taking the complement of a relation, whereas in stability theory the exact choice of
signature does not matter and all first-order definable sets are a priori equal.
It now appears that on graph classes such that every (not necessarily induced) subgraph
of a member is again in the class, tameness notions from stability theory,
finite model theory and also algorithmic graph theory can be compared in a meaningful though somewhat coarse way.
For subgraph-closed classes of graphs we show that nowhere density is equivalent to stability and to
dependence (not having the independence property). Equivalence of stability and dependence in this context
is somewhat surprising, although it is well known under the stability theoretic assumption of simplicity.

Stability and the independence property are two key dividing lines in Shelah's classification theory programme
for infinite model theory.~\cite{shelah90,math/9608205}
Stable theories do not have a formula that codes an infinite linear order. This strong and robust tameness property is the key
assumption on which Shelah originally built his monumental machinery of stability theory.
At the other end, theories with the independence property have a formula that can code every subset of some infinite set.
Stability theory has recently made advances into general theories without the independence property, but much remains to be done.
The independence property is a strong wildness property, although some theories with the independence property,
such as that of the random graph, are actually very easy to understand from a stability-theoretic point of view.
A formula has the independence property if and only if it has infinite Vapnik-Chervonenkis dimension --
a key wildness notion in computational learning theory.~\cite{vc,laskowski}

We hope for further translations between notions of tameness in stability theory
and notions of tameness in combinatorial graph theory. This should allow us to identify well-behaved graph classes
with good algorithmic properties.
Moreover, we hope that these translations can ultimately be refined and extended to more general contexts such as arbitrary classes
of relational structures.

\section{Shallow graph minors and nowhere density}

In this paper graphs are undirected, without loops or multiple edges, and not necessarily finite.
From the point of view of model theory, such a graph is a relational structure $G$
with an irreflexive and symmetric binary relation $E^G$.
For the standard notions of graph theory we refer the reader to Diestel's book~\cite{diestel}.

$H$ is a \emph{minor} of $G$ if there is a subgraph $U\subseteq G$ (not necessarily an induced subgraph) and an equivalence relation
$\epsilon$ on $U$ with connected classes, such that $H\cong U/\epsilon$, i.e.{} $H$ is the 
result of contracting each $\epsilon$-class
to a single vertex.
$H$ is an \emph{$r$-minor} of $G$ if each $\epsilon$-equivalence class contains a vertex from which
the other vertices have distance at most~$r$.

$H$ is a \emph{topological minor} of $G$ if there is a subgraph $U\subseteq G$ (not necessarily an induced subgraph)
and an equivalence relation $\epsilon$ on $E^U$, i.e.{} on the edges,
such that each $\epsilon$-class is a path whose interior vertices all have degree~2,
and $H\cong U/\epsilon$, i.e.{} $H$ is the result of contracting each $\epsilon$-class to a single edge.
$H$ is a \emph{topological $r$-minor} of $G$ if moreover each $\epsilon$-equivalence class consists of at most $2r+1$ edges.
In other words, up to isomorphism the vertices of a topological ($r$-)minor $H$ of $G$ form a subset of the vertices of $G$, and the edges
of $H$ correspond to pairwise internally vertex disjoint paths in $G$ (of length at most $2r+1$),
whose interior points avoid~$H$.

In the following, we will consider isomorphism-closed classes $\mathcal C$ of graphs. A (topological) ($r$-)minor of $\mathcal C$
is a (topological) ($r$-)minor of a member of $\mathcal C$, respectively.
We write $\mathcal C\minors r$ for the class of $r$-minors of $\mathcal C$, and 
$\mathcal C\tminors r$ for the class of topological $r$-minors of $\mathcal C$.
In particular, $\mathcal C\minors 0=\mathcal C\tminors 0$
is the class of all graphs isomorphic to a subgraph of a member of $\mathcal C$.
Also note $\mathcal C\tminors r\subseteq\mathcal C\minors r$,
$(\mathcal C\minors r)\minors s \subseteq \mathcal C\minors(r+s)$ and
$(\mathcal C\tminors r)\tminors s \subseteq \mathcal C\tminors(r+s)$.

Nešetřil and Ossona de Mendez proved that as $r$ goes to infinity, there are only three possible
asymptotic behaviours for the growth of the number of edges of finite $r$-minors,
or equivalently finite topological $r$-minors, in terms of their vertex counts:
finitely bounded, linear, or quadratical.
\begin{fact}
\[
\lim_{r\to\infty} \limsup_{\substack{H\in\mathcal C\minors r\\H \textup{finite}\\\left|H\right|\to\infty}} \frac{\log \left\| H \right\|}{\log \left| H \right|} \quad=\quad
\lim_{r\to\infty} \limsup_{\substack{H\in\mathcal C\tminors r\\H \textup{finite}\\\left|H\right|\to\infty}} \frac{\log \left\| H \right\|}{\log \left| H \right|}
\quad\in\quad \{0,1,2\},
\]
where $\left\|H\right\|$ and $\left|H\right|$ are the edge count and vertex count of $H$, respectively.
Moreover, the quadratic case (right-hand side~2) is equivalent to the statement that for some $r$ there is no finite upper bound
on the sizes of cliques that occur as $r$-minors of $\mathcal C$,
or equivalently as topological $r$-minors.
\end{fact}
They called $\mathcal C$ \emph{nowhere dense} in the linearly bounded case, i.e.\ when for every $r$
there is a finite upper bound for the sizes of cliques that occur as $r$-minors
(or, equivalently, topological $r$-minors) of $\mathcal C$.~\cite{nwd}
If $\mathcal C$ is nowhere dense, then so is every subclass of every class of the form
$\mathcal C\minors r$.\footnote{Ne\v set\v ril and Ossona de Mendez only consider
classes of finite graphs, but the definitions and results carry over in a straightforward way to classes of
not necessarily finite graphs.}

An \emph{$m$-clique} is a complete graph on $m$ vertices, denoted by $K_m$. By
$K_m^r$ we denote the result of subdividing each edge of the $m$-clique $K_m$ exactly $r$~times.
Essentially following Podewski and Ziegler~\cite{sg}, we call $\mathcal C$ \emph{superflat} if for every $r$ there is an $m$
such that $K_m^r$ does not occur as a subgraph of a member of $\mathcal C$.
Using the finite Ramsey theorem, it is easy to see:

\begin{remark}\label{remark:nwdsuperflat}
Let $\mathcal C$ be a class of graphs.
$\mathcal C$ is nowhere dense if and only if $\mathcal C$ is superflat.
\end{remark}

\section{Stability of graphs}

Graphs and digraphs are examples of relational structures in the sense of first-order logic and model theory.
Since we will later treat coloured digraphs in this framework, it is worth introducing some of the terminology
in its general form. 
A \emph{relational signature} is a set $\sigma$ of relation symbols. Every relation symbol
$R\in\sigma$ has an associated non-negative integer $\textup{ar}(R)$, its \emph{arity}.
A \emph{$\sigma$-structure} $M$ consists of a set $U$ called the \emph{universe} or \emph{underlying set}, and a
relation $R^M \subseteq U^{\textup{ar}(R)}$ for every $R\in\sigma$.\footnote{Nullary relation symbols act
syntactically and semantically like the variables of propositional logic, encoding Boolean variables within structures.
Every nullary relation is a subset of the 1-element set which has the 0-tuple as its only element.
Such a relation is true if and only if it is non-empty.
Some authors exclude nullary relation symbols from the definition, but they may turn out useful in our context.}
For the standard notions of model theory see the book by Hodges~\cite{hodges}.
An undirected graph is an  $\{E\}$-structure $G$ with a single, binary relation $E^G$ 
that is irreflexive and symmetric.

The formulas of \emph{first order logic} are built in the usual way from 
variables $x,y,z,x_1,\ldots$, the equality symbol $=$,
the relation symbols in $\sigma$, the Boolean connectives $\wedge, \vee, \neg,
\rightarrow$, and the quantifiers $\forall,\exists$ ranging over the universe
of the structure. A \emph{free} variable in a first-order formula is a variable
$x$ that is not within the scope of a quantifier $\forall x$ or $\exists x$.
The notation $\phi(x_1,\ldots,x_n)$ indicates that all free variable of the formula $\phi$
are among $x_1,\ldots ,x_n$. For a formula $\phi(x_1,\ldots,x_n)$, a structure
$M$ and elements $a_1,\ldots,a_n$ of the universe of $M$ we write 
$M\models\phi(a_1,\ldots,a_n)$ to say that $M$ satisfies $\phi$ if the
variables $x_1,\ldots ,x_n$ are interpreted by the elements $a_1,\ldots,a_n$, respectively.

Let $\mathcal C$ be a class of structures of a fixed signature.
A first-order formula $\phi(\bar x,\bar y)$ is said to have the \emph{order property} with respect to~$\mathcal C$
if it has the \emph{$n$-order property} for all $n$, i.e.\ if for every $n$
there exist a structure $M\in\mathcal C$ and tuples $\bar a_0,\ldots,\bar a_{n-1},\bar b_0,\ldots,\bar b_{n-1}\in M$
such that $M\models\phi(\bar a_i,\bar b_j)$ holds if and only if $i<j$.
A class $\mathcal C$ of structures is called \emph{stable} if there is no such formula with respect to~$\mathcal C$.
It is easy to see that $\mathcal C$ is stable if and only if there is no
formula $\psi(\bar u, \bar v)$ with $|\bar u|=|\bar v|$, such that for every $n$
there exist a structure $M\in\mathcal C$ and tuples $\bar c_0,\ldots,\bar c_{n-1}\in M$ such that
$M\models\psi(\bar c_i,\bar c_j)$ holds if and only if $i\leq j$, i.e.{} $\psi$ orders the tuples linearly.

Stability and the ($n$-)order property come from stability theory~\cite{shelah90,math/9608205},
where they are defined for the class of models of a complete first-order theory.  
A single structure $M$ is called stable if $\{M\}$ is stable.
This is equivalent to requiring that the class of all structures elementarily
equivalent to $M$ be stable, and so our notion of stability generalises the usual one.
In this paper we are primarily interested in applying the concept to
classes of finite graphs.

An \emph{interpretation} $I$ of a class $\mathcal C$ of structures of a fixed relational
signature in a class $\mathcal D$ of structures of another fixed signature
is given by the following data.
We have a first-order formula $\delta_I(\bar y)$ in the signature of $\mathcal D$
and for each structure $A\in\mathcal C$ a structure $I(A)\in\mathcal D$.
For $B\in\mathcal D$ let $\delta_I^B$ denote
the set of tuples $\bar b\in B$ for which $B\models\delta_I(\bar b)$ holds.
For each $A\in\mathcal C$ there is a surjective map $s_I^A\colon\delta_I^{I(A)}\to A$.
For each first-order formula $\phi(x_0,\ldots, x_{n-1})$ in the signature of $\mathcal C$
there is a first-order formula $\phi_I(\bar y_0,\ldots,\bar y_{n-1})$ in
the language of $\mathcal D$, not depending on $A$, such that for all $\bar b_0,\ldots,\bar b_{n-1}\in I(A)$
we have $A\models\phi(s_I^A(\bar b_0),\ldots,s_I^A(\bar b_{n-1}))$ if and only if
$I(A)\models\phi_I(\bar b_0,\ldots,\bar b_{n-1})$.
It is enough to find formulas $\phi_I$ for atomic formulas $\phi$.

In this paper we will only construct interpretations in which $\bar y$ has length~1 and the maps $s_I^A$ are bijective.

\begin{remark}
If $\mathcal C$ is interpretable in $\mathcal D$ and $\mathcal D$ is stable,
then so is $\mathcal C$.
\end{remark}

The notion of superflatness was originally introduced by Podewski and Ziegler as a simple sufficient
condition for stability of infinite graphs. A graph $G$ is \emph{superflat} in their sense if and only if
$\{G\}$ is superflat.

\begin{fact}[Podewski, Ziegler \cite{sg}]\label{fact:superflat}
Every superflat graph $G$ is stable.
\end{fact}

Every subgraph of a superflat graph is superflat, but every graph is a subgraph of a stable graph (a clique).
Therefore the converse of Fact~\ref{fact:superflat} does not hold.

\begin{lemma}\label{lemma:superflat}
Let $\mathcal C$ be a class of graphs.
If $\mathcal C$ is superflat, then $\mathcal C$ is stable.
\end{lemma}
\begin{proof}
The basic idea is simply to put all members of $\mathcal C$ together into
a single graph $G(\mathcal C)$ whose vertex and edge sets are the disjoint unions of the vertex and
edge sets of the members of~$\mathcal C$.
If $\mathcal C$ is superflat, then so is $G(\mathcal C)$, which is therefore a stable graph
by Fact~\ref{fact:superflat}.
Any quantifier-free formula $\phi(\bar x, \bar y)$ can be read in $\mathcal C$ or in $G(\mathcal C)$.
It is easy to see that if $\phi(\bar x,\bar y)$ has the order property with respect to $\mathcal C$,
then the same is true with respect to $G(\mathcal C)$.
But unless all graphs in $\mathcal C$ are connected of bounded diameter, $\mathcal C$ is not interpretable
in $G(\mathcal C)$ because there is no suitable formula $\delta_I$.

From every graph $A\in\mathcal C$ we derive a graph $A'$ of diameter at most four, as follows.
The vertices of $A'$ are the vertices of $A$ together with one new vertex each for every edge
of $A$, as well as three additional vertices $t, t_1, t_2$. For every edge $(a,b)\in E^A$ of~$A$,
$A'$ contains the edges $(a,e),(e,b)\in E^{A'}$, where $e=\{a,b\}$. The remaining edges of $A'$ are
$(t,t_1), (t,t_2), (t_1,t_2)$ as well as one edge each connecting $t$
with every vertex of $A$.
Note that (assuming for simplicity that $A$ is non-empty) $t\in A'$ is definable by a first-order
formula, as the only vertex of degree $>2$ that is part of a 3-cycle. Hence the vertex set of
$A$ is also definable, as the set of vertices that are adjacent to $t$ but are not part of
a 3-cycle.
Transforming every graph $A\in\mathcal C$ in this way, we get a new class $\mathcal C'$.
It is easy to see that $\mathcal C'$ is again superflat, and so is $G(\mathcal C')$.
Hence $G(\mathcal C')$ is stable by Fact~\ref{fact:superflat}.

In $G(\mathcal C')$, we can again identify the vertices of type $t$ and the vertices corresponding
directly to vertices in a graph in $\mathcal C$ in first order. On the latter set of vertices,
the relation given by $a$ and $b$ being both connected to a point of type $t$ is an equivalence
relation, again definable in first order. This allows us to translate every first-order formula
$\phi(\bar x,\bar y)$ into a formula $\phi'(\bar x,\bar y)$ such that for $\bar a,\bar b\in A\in\mathcal C$,
$A\models\phi(\bar a,\bar b)$ if and only if $G(\mathcal C')\models\phi'(\bar a,\bar b)$.

Since $G(\mathcal C')$ is stable, $\phi'(\bar x,\bar y)$ does not have the order property
with respect to $G(\mathcal C')$, hence $\phi(\bar x, \bar y)$ does not have the order property
with respect to $\mathcal C$. Since this holds for arbitrary $\phi$, $\mathcal C$ is stable.
\end{proof}

\section{Stability of coloured digraphs}

In this section we extend Lemma~\ref{lemma:superflat} to classes of vertex- and edge-coloured directed graphs.
More precisely, we extend it to relational structures where all relation symbols are at most binary.
We call such relational structures \emph{coloured digraphs}.

By a \emph{coloured digraph} we will understand a relational structure whose relation symbols are at most binary.
The \emph{underlying graph} or \emph{Gaifman graph} $\underline M$ of a relational structure $M$ is the graph with vertices the
elements of $M$ and edges all pairs $(a,b)$ such that $a\not=b$ and $a$ and $b$ appear in an instance of a relation of $M$ together.
(I.e. $(a,b)\in E^{\underline M}$ if and only if $a\not=b$ and there exist a relation symbol $R$ and $c_1\ldots c_n\in R^M$
such that $a=c_i$, $b=c_j$ for some $i,j$.) For every class $\mathcal C$ of structures $M$ we let $\underline{\mathcal C}$
be the class of underlying graphs $\underline M$.
For a class of stuctures $\mathcal C$, 
the combinatorial complexity of the graphs in $\underline{\mathcal C}$ is a good 
indication for the computational complexity of algorithmic problems on $\mathcal C$.
This has been exploited in various areas such as complexity theory, database theory,
algorithmic graph theory
and finite model theory. Here we will only consider
underlying graphs of coloured digraphs, in which case 
the construction amounts to forgetting the colours, loops, and edge directions.

\begin{lemma}\label{lemma:graphdigraph}
Every class $\mathcal C$ of coloured digraphs of a fixed countable signature can be interpreted in a class $\mathcal C'$
of undirected graphs such that $\mathcal C'$ is superflat if and only if $\underline{\mathcal C}$ is superflat.
\end{lemma}
\begin{proof}
We enumerate the binary relation symbols in the signature of $\mathcal C$ as $R_1$, $R_2$, $R_3$,~$\ldots$,
the unary relation symbols as $P_1$, $P_2$, $P_3$,~$\ldots$ and the nullary relation symbols as $A_1$, $A_2$, $A_3$,~$\ldots$.
For a single coloured digraph~$G$, we define the following graph~$G'$.

\setlength{\unitlength}{0.9mm}
\begin{figure}[bh]
\begin{picture}(70,80)(+28,5)
\newsavebox{\spot}
\savebox{\spot}(1,1)[bl]{\circle*{1.5}}
\put(30,30){\usebox{\spot}}
\put(30,60){\usebox{\spot}}
\put(60,30){\usebox{\spot}}
\put(60,60){\usebox{\spot}}
\put(90,30){\usebox{\spot}}
\put(30,30){\line(0,1){30}}
\put(30,30){\line(1,0){60}}
\put(30,60){\line(1,0){30}}
\put(60,30){\line(0,1){30}}
\put(60,30){\circle{4}}
\put(60,60){\circle{4}}
\put(30,60){\line(-1,-2){3}}
\put(30,60){\line(1,-2){3}}
\put(60,60){\line(-1,-2){3}}
\put(60,60){\line(1,-2){3}}
\put(60,30){\line(-1,2){3}}
\put(60,30){\line(1,2){3}}
\put(60,30){\line(-2,-1){6}}
\put(60,30){\line(-2,1){6}}
\put(60,60){\line(-2,-1){6}}
\put(60,60){\line(-2,1){6}}
\put(90,30){\line(-2,-1){6}}
\put(90,30){\line(-2,1){6}}
\end{picture}
\hspace{\stretch{1}}
\begin{picture}(90,80)(+3,5)
\savebox{\spot}(1,1)[bl]{\circle*{1.5}}
\put(30,30){\usebox{\spot}}
\put(30,60){\usebox{\spot}}
\put(60,30){\usebox{\spot}}
\put(60,60){\usebox{\spot}}
\put(90,30){\usebox{\spot}}
\put(30,30){\line(0,1){30}}
\put(30,30){\line(1,0){60}}
\put(30,60){\line(1,0){30}}
\put(60,30){\line(0,1){30}}
\put(30,40){\usebox{\spot}}
\put(30,50){\usebox{\spot}}
\put(60,40){\usebox{\spot}}
\put(60,50){\usebox{\spot}}
\put(30,40){\usebox{\spot}}
\put(40,30){\usebox{\spot}}
\put(50,30){\usebox{\spot}}
\put(40,60){\usebox{\spot}}
\put(50,60){\usebox{\spot}}
\put(70,30){\usebox{\spot}}
\put(80,30){\usebox{\spot}}
\put(30,60){\line(-1,1){20}}
\put(20,70){\line(0,1){10}}
\put(10,70){\line(0,1){10}}
\put(20,70){\line(-1,0){10}}
\put(20,80){\line(-1,0){10}}
\put(10,70){\line(1,1){10}}
\put(20,70){\usebox{\spot}}
\put(10,70){\usebox{\spot}}
\put(20,80){\usebox{\spot}}
\put(10,80){\usebox{\spot}}
\put(60,60){\line(-1,1){20}}
\put(50,70){\line(0,1){10}}
\put(40,70){\line(0,1){10}}
\put(50,70){\line(-1,0){10}}
\put(50,80){\line(-1,0){10}}
\put(40,70){\line(1,1){10}}
\put(50,70){\usebox{\spot}}
\put(40,70){\usebox{\spot}}
\put(50,80){\usebox{\spot}}
\put(40,80){\usebox{\spot}}
\put(30,30){\line(-1,-1){20}}
\put(20,20){\line(0,-1){10}}
\put(10,20){\line(0,-1){10}}
\put(20,20){\line(-1,0){10}}
\put(20,10){\line(-1,0){10}}
\put(10,20){\line(1,-1){10}}
\put(10,20){\usebox{\spot}}
\put(20,20){\usebox{\spot}}
\put(20,10){\usebox{\spot}}
\put(10,10){\usebox{\spot}}
\put(60,30){\line(-1,-1){20}}
\put(50,20){\line(0,-1){10}}
\put(40,20){\line(0,-1){10}}
\put(50,20){\line(-1,0){10}}
\put(50,10){\line(-1,0){10}}
\put(40,20){\line(1,-1){10}}
\put(40,20){\usebox{\spot}}
\put(50,20){\usebox{\spot}}
\put(50,10){\usebox{\spot}}
\put(40,10){\usebox{\spot}}
\put(90,30){\line(-1,-1){20}}
\put(80,20){\line(0,-1){10}}
\put(70,20){\line(0,-1){10}}
\put(80,20){\line(-1,0){10}}
\put(80,10){\line(-1,0){10}}
\put(70,20){\line(1,-1){10}}
\put(70,20){\usebox{\spot}}
\put(80,20){\usebox{\spot}}
\put(80,10){\usebox{\spot}}
\put(70,10){\usebox{\spot}}
\put(30,50){\line(-1,0){10}}
\put(20,50){\usebox{\spot}}
\put(50,30){\line(0,1){10}}
\put(50,40){\usebox{\spot}}
\put(50,60){\line(0,-1){10}}
\put(50,50){\usebox{\spot}}
\put(60,40){\line(1,0){10}}
\put(70,40){\usebox{\spot}}
\put(60,50){\line(1,0){10}}
\put(70,50){\usebox{\spot}}
\put(80,30){\line(0,1){10}}
\put(80,40){\usebox{\spot}}
\put(60,20){\line(0,1){50}}
\put(60,20){\usebox{\spot}}
\put(60,70){\usebox{\spot}}
\put(80,80){\usebox{\spot}}
\put(90,80){\usebox{\spot}}
\put(80,70){\usebox{\spot}}
\put(90,70){\usebox{\spot}}
\put(90,60){\usebox{\spot}}
\put(80,80){\line(1,0){10}}
\put(80,70){\line(1,0){10}}
\put(80,80){\line(0,-1){10}}
\put(90,80){\line(0,-1){20}}
\end{picture}

\caption{Proof of Lemma~\ref{lemma:graphdigraph}. Example of $G$ with one unary and one binary relation
(depicted by circles and arrows, respectively), and corresponding~$G'$.
The small connected component in the top right corner consists of the vertices $d,d_0,d_1,d_2,d_3$
and is recognisable as the only chordless 4-cycle in $G'$.
As there are no nullary relation symbols in the signature, it serves no real purpose in this particular example.}
\end{figure}
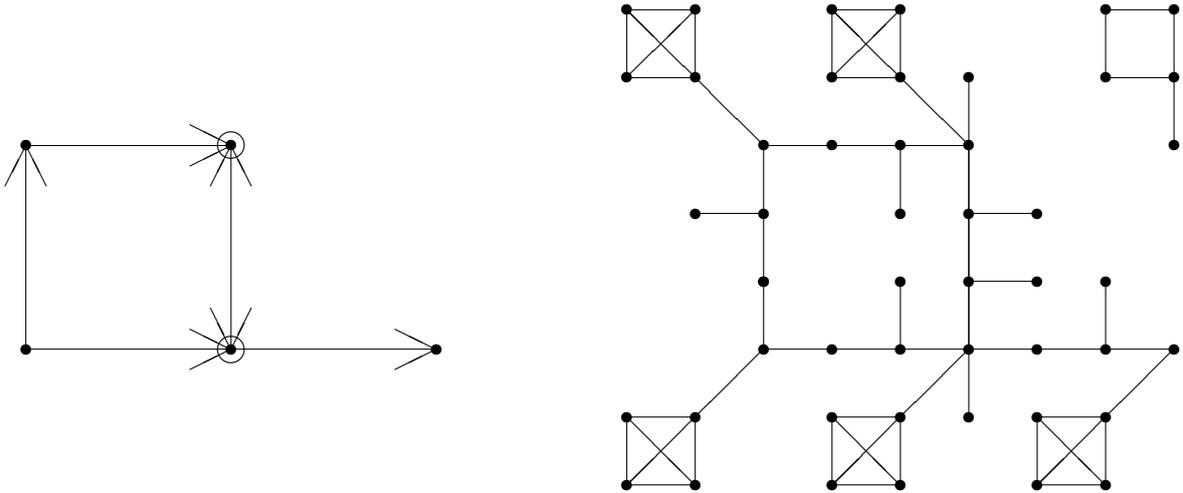

All vertices of $G$ are also vertices of $G'$.
Moreover, for every vertex $a$ of $G$ there are also four new vertices $a_0,a_1,a_2,a_3$ of $G'$ as well
as edges $(a,a_0)$, $(a_0,a_1)$, $(a_0,a_2)$, $(a_0,a_3)$, $(a_1,a_2)$, $(a_1,a_3)$, $(a_2,a_3)\in E^{G'}$. The 4-cliques will allow us to pick out the vertices of $G$ inside $G'$
by means of a first-order formula, since there will be no other 4-cliques in $G'$ and the 4-clique vertices will have no further
connections.
For every vertex $a\in P_i^G$ we add to $G'$ new vertices $c_1,\ldots,c_i$ and the path
$(a,c_1)$, $(c_1,c_2)$, $\ldots$, $(c_{i-1},c_i)\in E^{G'}$.

For every unordered pair $\{a,b\}$ of vertices from $G$ (we allow $a=b$) such that $(a,b)$ or $(b,a)$ appears in one of the binary relations of $G$,
$G'$ contains new vertices $c_{ba}$ and $c_{ab}$ as well as edges $(a,c_{ba})$, $(c_{ba}$, $c_{ab})$, $(c_{ab},b)\in E^{G'}$.
I.e.{} any two vertices between which there is a directed edge are connected by an undirected path of length~3.
For every directed edge $(a,b)\in R_i^G$ we add to $G'$ new vertices $c_1,\ldots,c_i$ and the path
$(c_{ab},c_1)$, $(c_1,c_2),\ldots,(c_{i-1},c_i)\in E^{G'}$.

Finally, to treat the nullary relations we also add new points $d, d_0, d_1, d_2, d_3$ and edges
$(d,d_0), (d_0,d_1), (d_1,d_2), (d_2,d_3), (d_3, d_1)\in E^{G'}$.
For each $i$ such that $G\models A_i$ we attach a new path $c_1,c_2,\ldots,c_i$ to $d$,
similar to the unary and binary cases.

We get $\mathcal C'$ from $\mathcal C$ by treating each $G\in\mathcal C$ in this way.
It is easy to see that $\mathcal C$ can be interpreted in $\mathcal C'$ and that $\mathcal C'$ is superflat if and only
if $\underline{\mathcal C}$ is superflat.
\end{proof}

\begin{theorem}\label{theorem:superflat}
Let  $\mathcal C$ be a class of coloured digraphs of a fixed signature.
If  $\underline{\mathcal C}$ is superflat, then $\mathcal C$ is stable.
\end{theorem}
\begin{proof}
Since every formula contains only a finite part of the signature, we may assume that the signature is finite.
By Lemma~\ref{lemma:graphdigraph} we can interpret $\mathcal C$ in a superflat class $\mathcal C'$ of graphs.
By Lemma~\ref{lemma:superflat}, $\mathcal C'$ is stable.
It follows that $\mathcal C$ is also stable.
\end{proof}

\section{Independence property}

A first-order formula $\phi(\bar x,\bar y)$ is said to have the \emph{independence property} with respect to~$\mathcal C$
if it has the \emph{$n$-independence property} for all $n$, i.e.\ if for every $n$
there exist a structure $M\in\mathcal C$ and tuples $\bar a_0,\ldots,\bar a_{n-1}\in M$ and
$\bar b_J\in M$ for all $J\subseteq\{0,1,\ldots,n-1\}$ such that $M\models\phi(\bar a_i,\bar b_J)$ holds if and only if $i\in J$.
$\mathcal C$ is said to be \emph{dependent} or to have \emph{NIP}
 if no formula has the independence property with respect to~$\mathcal C$.
One can show that $\phi(\bar u,\bar v)$ has the independence property
if and only if the `opposite' formula $\phi(\bar v,\bar u)$ (i.e.\ really the same formula, but listing the variables differently) has it.
It is easy to see that every formula with the independence property has the order property. Therefore every stable class is dependent.
See \cite{nipintro} for more on the independence property and its relation to the order property.

Like stability, the independence property comes from stability theory~\cite{shelah90,math/9608205}, and is originally only defined
for first-order theories.
Again, a single structure $M$ is called dependent if $\{M\}$ is dependent or,
equivalently, if the class of all structures elementarily
equivalent to $M$ is dependent.

\begin{lemma}\label{lemma:NIPsuperflat}
Let $\mathcal C$ be a subgraph-closed class of graphs.
If $\mathcal C$ is dependent, then $\mathcal C$ is superflat.
\end{lemma}

\begin{proof}
Suppose $\mathcal C$ is not superflat, i.e.\ for some $r$, every $K_m^r$
occurs as a subgraph of a member of $\mathcal C$.
Since the following graph $A_m$ is a subgraph of $K_{m+2^m}^r$,
it also occurs as a subgraph of a member of $\mathcal C$,
hence is itself a member of $\mathcal C$ (up to isomorphism).
$A_m$ has vertices $a_0,a_1,\ldots,a_{m-1}$ and $b_J$ for each $J\subseteq\{0,1,\ldots,m-1\}$
as well as additional vertices that appear in the following.
For any $i\in\{0,1,\ldots,m-1\}$ and any $J\subseteq\{0,1,\ldots,m-1\}$
such that $i\in J$, there is a path of length $r+1$ from $a_i$ to $b_J$.
The interior parts of these paths are pairwise disjoint and disjoint from the set
of vertices $a_i$ and $b_J$. There are no further vertices or edges.

Let $\phi(x,y)$ be the formula that says that there is a path of length $r+1$ from $x$ to $y$.
Since $A_m\models\phi(a_i,b_J)$ if and only if $i\in J$,
the family of graphs $A_m$ witnesses that $\phi(x,y)$ has the independence property with respect to $\mathcal C$.
So $\mathcal C$ is not dependent.
\end{proof}

We will call a class $\mathcal C$ of relational structures \emph{monotone} if whenever $M\to N$ is an injective homomorphism
and $N\in\mathcal C$, we also have $M\in\mathcal C$. In other words, a monotone class is closed under isomorphism
and `non-induced substructures', the natural generalisation of  non-induced subgraphs to arbitrary signatures.
Putting all the previous results together, we have the following theorem.

\begin{theorem}
Let  $\mathcal C$ be a monotone class of coloured digraphs of a fixed finite signature.
The following conditions are equivalent.
\begin{enumerate}
\item $\underline{\mathcal C}$ is nowhere dense.
\item $\underline{\mathcal C}$ is superflat.
\item $\mathcal C$ is stable.
\item $\underline{\mathcal C}$ is stable.
\item $\mathcal C$ is dependent.
\item $\underline{\mathcal C}$ is dependent.
\end{enumerate}
\end{theorem}

\begin{proof}
The first two conditions are equivalent by Remark~\ref{remark:nwdsuperflat} and imply the third
by Theorem~\ref{theorem:superflat}.
$3\Rightarrow5$ and $4\Rightarrow6$ because every stable formula is dependent.
$3\Rightarrow4$ and $5\Rightarrow6$ because $\underline{\mathcal C}$ is interpretable in $\mathcal C$.
Finally, $6\Rightarrow1$ by Lemma~\ref{lemma:NIPsuperflat} because
$\underline{\mathcal C}$ is closed under subgraphs.
\end{proof}

As a corollary to the proof, we see that to check stability or dependence of a monotone class of coloured digraphs
it is sufficient to look at formulas of the form $\phi(x,y)$ with single variables $x$ and $y$.
This is not true in general.

The condition that $\mathcal C$ be monotone is crucial. The class of all cliques is stable
but not nowhere dense / superflat. It is also not hard to code the class of linear orders in a class
of graphs that is dependent but not stable.
Further it is crucial for the equivalence of stability and NIP that $\mathcal C$ is a class of relational structures
(i.e.{} the signature does not contain function symbols), since
the class of all linear orders coded by the binary function $\min$ is dependent but not stable.
Also note that the underlying graph of a structure with a binary function symbol is always complete.
Finally, in a signature with infinitely many binary relation symbols $\underline{\mathcal C}$ may
not be interpretable in $\mathcal C$, and in fact a monotone class $\mathcal C$ of such a signature
may be stable even though $\underline{\mathcal C}$ is not.

\section{Excluded topological minors and strong stability}

Except for the complication that comes from the lack of a superstability notion for general classes
of graphs, this section is essentially a corollary of the previous one and of a follow-up to the paper
by Podewski and Ziegler.

An obvious way of strengthening nowhere density is to require that for some~$m$,
the $m$-clique $K_m$ is not a minor at all or, equivalently,
that there be some finite graph $G$ which does not occur as a minor of a graph in $\mathcal C$.
While the definition of nowhere density stays the same if we add `topological', this is not
true for this strengthening.
As the class of all graphs of maximum degree three shows, a class can omit a finite graph as
topological minor (in this case the 4-star) but still have all finite graphs as minors.

Following Herre, Mekler and Smith~\cite{ssg}, we call a class $\mathcal C$ of graphs \emph{ultraflat}
if for some~$m$, $K_m$ is not a topological minor of any member of~$\mathcal C$.\footnote{Like
Podewski and Ziegler, these authors only considered a single, infinite graph.}

\begin{remark}
 A class of graphs is ultraflat if and only if it omits a finite topological minor.
\end{remark}

\begin{fact}[Herre, Mekler, Smith \cite{ssg}]\label{fact:ultraflat}
 Every ultraflat graph is superstable.
\end{fact}

\begin{corollary}
Let $G$ be a coloured digraph. If $\underline G$ is ultraflat, then $G$ is superstable.
\end{corollary}
\begin{proof}
We apply the construction in the proof of Lemma~\ref{lemma:graphdigraph} to $G$
and observe that $G'$ is ultraflat if and only if $\underline G$ is ultraflat.
It follows that $G'$ is superstable. Since $G$ is interpretable in $G'$,
$G'$ is also superstable.
\end{proof}

We cannot use this result directly because we do not know a generalisation of superstability
that is appropriate for general classes of structures. 
However, the more recent notions of strongly dependent and strongly stable theories are easily generalised.
Essentially\footnote{We have to adapt the definition to fit it into the present context,
in the same way that we did earlier for the order and independence properties.} following Shelah~\cite{modest},
we call a class $\mathcal C$ of structures \emph{not strongly dependent} if there is an infinite
sequence of first-order formulas $\phi^0(\bar x,\bar y^0), \phi^1(\bar x,\bar y^1), \phi^2(\bar x,\bar y^2), \ldots$
in the signature of $\mathcal C$, such that for every $m$ there exist a structure $M\in\mathcal C$
and tuples $\bar b^i_j\in M$, where each $\bar b^i_j$ has the same length as $\bar y^i$,
as well as for every function $f\colon \{0,1,\ldots,m-1\}\to\{0,1,\ldots,m-1\}$
a tuple $\bar a_f\in M$ such that $M\models\phi^i(\bar a_f,\bar b^i_j)$ if and only if $j=f(i)$.
In other words, in the array
\[\begin{array}{ccccc}
\phi^0(\bar x,\bar b^0_0), & \phi^0(\bar x,\bar b^0_1), & \phi^0(\bar x,\bar b^0_2), & \ldots, & \phi^0(\bar x,\bar b^0_{m-1})\\
\phi^1(\bar x,\bar b^1_0), & \phi^1(\bar x,\bar b^1_1), & \phi^1(\bar x,\bar b^1_2), & \ldots, & \phi^1(\bar x,\bar b^1_{m-1})\\
\vdots & \vdots & \vdots & \ddots & \vdots\\
\phi^{m-1}(\bar x,\bar b^{m-1}_0), & \phi^{m-1}(\bar x,\bar b^{m-1}_1), & \phi^{m-1}(\bar x,\bar b^{m-1}_2), & \ldots, & \phi^{m-1}(\bar x,\bar b^{m-1}_{m-1}),
\end{array}\]
if we pick one formula from each row, there is always a tuple $\bar a\in M$ that makes precisely the chosen formulas true.
Since a formula $\phi(\bar x,\bar y)$ has the independence property for $\mathcal C$ if and only if
the constant sequence $\phi(\bar x,\bar y), \phi(\bar x,\bar y), \phi(\bar x,\bar y), \ldots$ witnesses that
$\mathcal C$ is not strongly dependent, every strongly dependent class is in fact dependent.
We call a class of structures \emph{strongly stable} if it is stable and strongly dependent.\footnote{Strongly
dependent first-order theories turn out to be those in which every type has finite weight, but it is not clear whether
anything like this can be expressed in our more general context.}

\begin{remark}
If $\mathcal C$ is interpretable in $\mathcal D$ and $\mathcal D$ is strongly stable or strongly dependent,
then so is $\mathcal C$.
\end{remark}

\begin{fact}\label{fact:superstable}
Every superstable structure is strongly stable.
\end{fact}

\begin{corollary}
Let  $\mathcal C$ be a class of coloured digraphs of a fixed signature.
If $\underline{\mathcal C}$ omits a finite topological minor, then $\mathcal C$ is strongly stable.
\end{corollary}

\begin{proof}
We will essentially repeat the proof of Theorem~\ref{theorem:superflat} with some patches.
We may assume that the signature is countable.
First note that $\mathcal C'$ as in Lemma~\ref{lemma:graphdigraph} is ultraflat provided $\mathcal C$ is ultraflat.
By a similarly straightforward adaptation of Lemma~\ref{lemma:superflat} (using Facts~\ref{fact:ultraflat} and~\ref{fact:superstable} instead of Fact~\ref{fact:superflat}),
$\mathcal C'$ is strongly stable.
Since $\mathcal C$ is interpretable in the strongly stable class $\mathcal C'$,
$\mathcal C$ is itself strongly stable.
\end{proof}

\section{Conclusion}

We have seen that tameness notions from combinatorial graph theory, finite model theory
and stability theory can be compared for classes of graphs, so long as they are closed under subgraphs.
The latter restriction is a rather natural one in the first two fields, but severe and unnatural from the 
point of view of stability theory, of which the observed collapse of stability and NIP may be a symptom.
This will probably make a transfer of ideas from stability theory to the other fields more straightforward than in
the other direction.
For some dividing lines of stability theory (superstability, simplicity) we have not found a
suitable generalisation.
It remains to be seen to what extent
parts of stability theory (indiscernibles, forking or splitting,
etc.) can be generalised
to the new context and whether they are of any relevance to the algorithmically
oriented fields.

Finally, we hope that bringing together the tools from the different fields will
make it easier to find a unifying combinatorial explanation for the algorithmically tame (or wild)
behaviour of many graph classes.

\bibliographystyle{hplain}
\bibliography{nwd}

\end{document}